\documentclass{my-aims}

\usepackage{amsmath}
\usepackage{paralist}
\usepackage{graphics} 
\usepackage{epsfig} 
\usepackage{graphicx}  
\usepackage{epstopdf}
\usepackage[colorlinks=true]{hyperref}

\hypersetup{urlcolor=blue, citecolor=red}
\usepackage{hyperref}

    \def\ppages{X--XX}

\newtheorem{theorem}{Theorem}[section]
\newtheorem{corollary}{Corollary}
\newtheorem{lemma}[theorem]{Lemma}
\newtheorem{proposition}{Proposition}
\theoremstyle{definition}
\newtheorem{definition}[theorem]{Definition}
\newtheorem{remark}{Remark}

\usepackage{cite}
\usepackage{array,multirow,makecell}
\setcellgapes{1pt}
\makegapedcells
\newcolumntype{R}[1]{>{\raggedleft\arraybackslash }b{#1}}
\newcolumntype{L}[1]{>{\raggedright\arraybackslash }b{#1}}
\newcolumntype{C}[1]{>{\centering\arraybackslash }b{#1}}
\usepackage{algorithm}
\usepackage{algorithmic}
\usepackage{caption}

\allowdisplaybreaks

\title[Optimality conditions of a SIR model with fractional derivatives]%
{Necessary optimality conditions of a reaction-diffusion SIR model with ABC fractional derivatives}

\author[M. R. Sidi Ammi, M. Tahiri and Delfim F. M. Torres]{}



\email{sidiammi@ua.pt, rachidsidiammi@yahoo.fr}
\email{my.mustafa.tahiri@gmail.com}
\email{delfim@ua.pt}

\thanks{$^*$ Corresponding author: M. R. Sidi Ammi}

\begin{document}

\maketitle

\centerline{\scshape Moulay Rchid Sidi Ammi$^{a*}$, Mostafa Tahiri$^a$ and Delfim F. M. Torres$^b$}
\medskip
{\footnotesize
\centerline{$^a$Department of Mathematics, AMNEA Group, Laboratory MAIS,}
\centerline{ Faculty of Sciences and Techniques, Moulay Ismail University of Meknes, Morocco}
\centerline{$^b$Center for Research and Development in Mathematics and Applications (CIDMA),}
\centerline{Department of Mathematics, University of Aveiro, 3810-193 Aveiro, Portugal.}
} 

\bigskip

\centerline{}


\begin{abstract}
The main aim of the present work is to study and analyze 
a reaction-diffusion fractional version of the SIR epidemic 
mathematical model by means of the non-local and non-singular 
ABC fractional derivative operator with complete memory effects.  
Existence and uniqueness of solution for the proposed fractional
model is proved.  Existence of an optimal control is also established. 
Then, necessary optimality conditions are derived. As a consequence, 
a characterization of the optimal control is given. Lastly, numerical 
results are given with the aim to show the effectiveness of the 
proposed control strategy, which provides significant results using 
the AB fractional derivative operator in the Caputo sense, comparing 
it with the classical integer one. The results show the importance 
of choosing very well the fractional characterization of the
order of the operators.
\end{abstract}

\maketitle
	

\section{Introduction}
\label{sec1}

Fractional derivatives give rise to theoretical models that allow 
a significant improvement in the fitting of real data when 
compared with analogous classical models \cite{MR3719831}. 
For real data of Florida Department of Health from September 2011 to July 2014, 
some authors conclude that the absolute error between the solutions 
obtained statistically and that of fractional models are smaller than those 
obtained by models of integer derivatives \cite{ref2}. In the fractional 
calculus literature, systems using fractional derivatives 
give a more realistic behavior \cite{ref1,ref3,ref4}. There exists 
many definitions of fractional derivative \cite{ref4}. Among the more well-known 
fractional derivatives, we can cite the Riemann--Liouville one. It is not always 
suitable for modeling physical systems, because the Riemann--Liouville derivative 
of a constant is not zero, and the initial conditions of associated Cauchy problems 
are expressed by fractional derivatives. Caputo fractional derivatives
offers another alternative, where the derivative of a constant is null 
and initial conditions are expressed as in the classical case 
of integer order derivatives \cite{ref3,ref4,ref5}. However, 
the kernel of this derivative has a singularity. Fractional derivatives 
that possess a non-singular kernel have aroused more interest from the scientific community. 
This is due to the non-singular memory of the Mittag--Leffler function 
and also to the non-obedience of the algebraic criteria of associativity 
and commutativity. The ABC fractional derivative is sometimes preferable 
for modeling physical dynamical systems, giving a good description 
of the phenomena of heterogeneity and diffusion at different scales 
\cite{ref6,ref7,ref8}.

Fractional calculus plays an important role in many areas 
of science and engineering. It also finds application 
in optimal control problems. The principle of mathematical theory 
of control is to determine a state and a control for a dynamic 
system during a specified period to optimize a given objective \cite{ref9}. 
Fractional optimal control problems have been formulated and studied 
as fractional problems of the calculus of variations. Some authors 
have shown that fractional differential equations are more accurate 
than integer-order differential equations, and that fractional controllers 
work better than integer-order controllers \cite{ref10,ref11,ref12,ref13}. 
In \cite{ref20}, Yuan et al. have studied problems of fractional optimal 
control via left and right fractional derivatives of Caputo.  
A numerical technique for the solution of a class of fractional 
optimal control problems, in terms of both Riemann--Liouville and  
Caputo fractional derivatives, is presented in \cite{ref15}. Authors in 
\cite{ref16,ref17} present a pseudo-state-space fractional optimal control 
problem formulation. Fixed and free final-time fractional optimal control problems 
are considered in \cite{ref18,ref19}. Guo \cite{ref21} formulates 
a second-order necessary optimality condition for fractional optimal control 
problems in the sense of Caputo. Optimal control of a fractional-order 
HIV-immune system, in terms of Caputo fractional derivatives, 
is discussed in \cite{ref22}. In \cite{ref23}, authors proposed 
a fractional-order optimal control model for malaria infection 
in terms of the Caputo fractional derivative. Optimal control 
of fractional diffusion equations has also been studied by several authors. 
For instance, in \cite{ref24}, Agrawal considers two problems, the simplest 
fractional variation problem and a fractional variational problem in Lagrange form. 
For both problems, the author developed Euler--Lagrange type necessary conditions, 
which must be satisfied for the given functional to have an extremum. In \cite{ref1}, 
authors prove necessary optimality conditions of a nonlocal thermistor problem 
with ABC fractional time derivatives.

Several infectious diseases confer permanent immunity against reinfection. 
This type of diseases can be modeled by the $SIR$ model. 
The total population $(N)$ is divided into three compartments with 
$N = S + I + R$, where $S$ is the number of susceptible 
(those able to contract the disease), $I$ is the number 
of infectious individuals (those capable of transmitting the disease), 
and $R$ is the number of individuals recovered (those who have recovered and become immune). 
Vaccines are extremely important and have been proved to be most effective 
and cost-efficient method of preventing infectious diseases, such as measles, 
polio, diphtheria, tetanus, pertussis, tuberculosis, etc. The study of fractional 
calculus with a non-singular kernel is gaining more and more attention. 
Compared with classical fractional calculus with a singular kernel, 
non-singular kernel models can describe reality more accurately, 
which has been shown recently in a variety of fields such as physics, 
chemistry, biology, economics, control, porous media, aerodynamics and so on. 
For example, extensive treatment and various applications of fractional 
calculus with non-singular kernel has been discussed in the works of 
Atangana and Baleanu \cite{ref7}, and Djida et al. \cite{djida}. 
It has been demonstrated that fractional order differential equations (FODEs) 
with non-singular kernels give rise to dynamic system models that are more accurately.

In this work, we consider an optimal control problem for the
reaction-diffusion SIR system with Atangana--Baleanu fractional 
derivative in the Caputo sense (the ABC operator). Our aim is 
to study the effect of non-local memory and vaccination strategies 
on the cost, needed to control the spread of infectious diseases.  
Our results generalize to the ABC fractional setting previous studies 
of classical control theory presented in \cite{ref27}.
The considered model there does not explain the influence of a complete memory 
of the system. For that, we extend such nonlinear system of first order differential  
equations to a fractional-order one in the ABC sense. We have further improved 
the cost and effectiveness of proposed control strategy during a given period of time.

This paper is organized as follows. Some important definitions related 
to the ABC fractional derivative operator and its properties are presented 
in Section~\ref{sec2}, while the underlying fractional reaction-diffusion 
SIR mathematical model is formulated in Section~\ref{sec3}. This led 
to the necessity of proving existence and uniqueness of solution 
to the proposed fractional model as well as existence of an optimal control. 
These results are extensively discussed in Sections~\ref{sec4} and \ref{sec5}. 
Section~\ref{sec6} is devoted to necessary optimality conditions. 
Interesting numerical tests, showing the importance of choosing very well the fractional
characterization $\alpha$, are given in Section~\ref{sec7}. Finally,  
conclusions of the present study are widely discussed in Section~\ref{sec8}.


\section{Preliminary results}
\label{sec2}

We now recall some properties on the Mittag--Leffler function
and the definition of ABC fractional time derivative.
First, we define the two-parameter Mittag--Leffler function
$E_{\alpha, \xi}(z)$, as the family of entire functions
of $z$ given by
$$
E_{\alpha, \xi}(z) = \sum_{k=0}^{\infty}
\frac{z^{k}}{\Gamma(k\alpha + \xi)},
\quad z \in \mathbb{C},
$$
where $\Gamma(\cdot)$ denotes the Gamma function
$$
\Gamma(z)= \int_{0}^{\infty} t^{z-1} e^{-t} dt,
\quad Re(z) > 0.
$$
Observe that the Mittag--Leffler function is a generalization 
of the exponential function: $E_{1, 1}(z)=e^{z}$.
For more information about the definition of fractional derivative 
in the sense of Atangana--Baleanu, the reader can see \cite{ref7,ref8}.

\begin{definition}
\label{def:1.1}
For a given function $g \in H^{1}(a, T)$, $T>a$,
the Atangana--Baleanu fractional derivative in Caputo sense,
shortly called the ABC fractional derivative
of $g$ of order $\alpha \in (0, 1)$ with base point $a$,
is defined at a point $t \in (a, T)$ by
\begin{equation}
\label{eq3bis}
^{abc}_{a}D_{t}^{\alpha} g(t)=\frac{B(\alpha)}{1-\alpha}
\int_{a}^{t} g'(\tau) E_{\alpha}[-\gamma (t-\tau)^{\alpha}] d\tau,
\end{equation}
where $\gamma= \frac{\alpha}{1-\alpha}$, $E_{\alpha, 1}=E_{\alpha}$
stands for the Mittag--Leffler function, and
$B(\alpha)= (1-\alpha) + \frac{\alpha }{\Gamma(\alpha)}$.
Furthermore, the Atangana--Baleanu fractional integral
of order $\alpha\in (0, 1)$ with base point $a$ is defined as
\begin{equation}
\label{eq4bis}
^{ab}_{a}I_{t}^{\alpha} g(t)= \frac{1-\alpha}{B(\alpha)} g(t)
+\frac{\alpha}{B(\alpha)\Gamma(\alpha)}
\int_{a}^{t} g(\tau) (t-\tau)^{\alpha-1} d\tau.
\end{equation}
\end{definition}

\begin{remark}
The usual ordinary derivative $\partial_{t}$ is obtained by letting 
$\alpha \rightarrow 1$ in \eqref{eq3bis}. If $\alpha =0, 1$ in \eqref{eq4bis},
then we get the initial function and the classical integral, respectively.
\end{remark}

\begin{definition}[See \cite{ref6}]
\label{def:1.1'}
For a given function $g \in H^{1}(a, T)$, $T>a$,
the backward Atangana--Baleanu fractional derivative in Caputo sense
of $g$ of order $\alpha \in (0, 1)$ with base point $T$,
is defined at a point $t \in (a, T)$ by
\begin{equation}
\label{eq3'bis}
^{abc}_{T}D_{t}^{\alpha} g(t)=-\frac{B(\alpha)}{1-\alpha}
\int_{t}^{T} g'(\tau) E_{\alpha}[-\gamma (t-\tau)^{\alpha}] d\tau.
\end{equation}
\end{definition}


\section{Model formulation}
\label{sec3}

The SIR model is one of the simplest compartmental models. It was first 
used by Kermack and McKendrick in 1927. It has subsequently been applied 
to a variety of diseases, especially airborne childhood diseases 
with lifelong immunity upon recovery, such as measles, mumps, rubella, and pertussis.
We assume that the populations are in a spatially homogeneous environment 
and their densities depend on space, reflecting the spacial spread of the disease.
Then, the model will be formulated as a system of reaction-diffusion equations.
In this section, we formulate an optimal control of a nonlocal fractional SIR 
epidemic model with parabolic equations and boundary conditions. We consider 
that the movement of the population depends on the time $t$ and the space $x$.
Furthermore, all susceptible vaccinates are transferred directly to the recovered class.

Let $Q_{T}=[0, T]\times\Omega$ and $\Sigma_{T}=[0, T]\times\partial\Omega$, 
where $\Omega$ is a fixed and bounded domain in $\mathbb{R}^2$ 
with smooth boundary $\partial\Omega$, and $[0,T]$ is a finite interval. 
The dynamic of the ABC fractional SIR system with control is given by
\begin{equation}
\label{eq1}
\begin{aligned}
^{abc}_{0}D{_t^{\alpha}}S(t,x) 
&= \lambda_1\Delta S(t,x)+\mu N(t,x)-\beta S(t,x) I(t,x)-dS(t,x)\\
&\quad-u(t,x)S(t,x),\\[2ex]
^{abc}_{0}D{_t^{\alpha}}I(t,x) 
&= \lambda_2\Delta I(t,x) +\beta S(t,x) I(t,x)-(d+r)I(t,x), \hspace{0.2cm} (t,x)\in Q_{T},\\[2ex]
^{abc}_{0}D{_t^{\alpha}}R(t,x) 
&= \lambda_3\Delta R(t,x) + r I(t,x)- dR(t,x) + u(t,x)S(t,x),
\end{aligned}
\end{equation}
with the homogeneous Neumann boundary conditions
\begin{equation}
\label{eq2}
\frac{\partial S(t,x)}{\partial \nu}
=\frac{\partial I(t,x)}{\partial \nu}
=\frac{\partial R(t,x)}{\partial \nu}=0, 
\quad (t,x)\in \Sigma_{T},
\end{equation}
and the following initial conditions of the three populations, 
which are considered positive for biological reasons:
\begin{equation}
\label{eq3}
S(0,x)=S_0, \quad I(0,x)=I_0 \quad 
\text{ and } \quad R(0,x)=R_0, \quad x\in\Omega.
\end{equation}
The positive constants $\mu$, $r$ and $d$  are respectively the birth rate, 
the recovery rate of the infective individuals and the natural death rate. 
Susceptible individuals acquire infection by the contact with individuals 
in the class $I$ at a rate $\beta SI$, where $\beta$ is the infection coefficient. 
Positive constants $\lambda_1$, $\lambda_2$, $\lambda_3$ denote the diffusion 
coefficients for the susceptible, infected and recovered individuals.  
The control $u$ describes the effect of vaccination. It is assumed that 
vaccination transforms susceptible individuals to recovered ones 
and confers them immunity.  The notation $\Delta=\frac{\partial^2}{\partial x^2}
+\frac{\partial^2}{\partial y^2}$ represents the usual Laplacian operator 
in two-dimensional space; $\nu$ is the outward unit normal vector 
on the boundary with $\frac{\partial}{\partial\nu}=\nu$; and 
$\nabla$ is the normal derivative on $\partial\Omega$. 
The no-flux homogeneous Neumann boundary conditions imply 
that model \eqref{eq1} is self-contained and there is a dynamic 
across the boundary, but there is no emigration.

Since the vaccination is limited and represents an economic burden, 
one important issue and goal is to know how much we should spend 
in vaccination to reduce the number of infections and, at the same time, 
save the cost of vaccination program. This can be mathematically 
interpreted by optimizing the following objective functional:
\begin{equation}
\label{eq4}
J(S,I,R,u)=\|I(t,x)\|_{L^2(Q_{T})}^{2} + \|I(T,\cdot)\|_{L^2(\Omega)}^{2} 
+ \theta\|u(t,x)\|_{L^2(Q_{T})}^{2},
\end{equation}
where $\theta$ is a weight constant for the vaccination control $u$, 
which belongs to the set 
\begin{equation}
\label{eq5}
U_{ad}=\{u\in L^\infty(Q_{T}); \|u\|_{L^\infty(Q_{T})}<1 
\quad \mbox{ and } \quad u>0\}
\end{equation}
of admissible controls. Let $y=(y_{1}, y_{2}, y_{3})= (S, I, R)$, 
$y^{0}=(y_{1}^{0}, y_{2}^{0}, y_{3}^{0})= (S_{0}, I_{0}, R_{0})$, 
$\lambda=(\lambda_{1}, \lambda_{2}, \lambda_{3})$, 
$\mathrm{L}^{2}(\Omega)= (L^{2}(\Omega))^{3}$
and $A$ be the linear diffusion operator defined by
$$
A: D(A) \subset \mathrm{L}^{2}(\Omega) 
\rightarrow \mathrm{L}^{2}(\Omega)
$$
$$
Ay = \lambda\Delta y = ( \lambda_{1} \Delta y_{1}, \lambda_{2} 
\Delta y_{2}, \lambda_{3} \Delta y_{3}), 
\quad \forall y \in D(A),
$$
where 
$$
D(A)= \left  \{  y=(y_{1}, y_{2}, y_{3}) \in (H^{2}(\Omega))^{3}, 
\frac{\partial y_{1}}{\partial \nu}
= \frac{\partial y_{2}}{\partial \nu}
= \frac{\partial y_{3}}{\partial \nu}=0,
\quad \mbox{a.e. } x \in \partial \Omega  \right \}.
$$
We also set
$$
f(y(t))= (f_{1}(y(t)), f_{2}(y(t)), f_{3}(y(t)))
$$
with
\[
\begin{cases}
f_{1}(y(t)) = & \mu (y_{1}+ y_{2}+ y_{3}) 
- \beta y_{1}y_{2}-d y_{1}-u y_{1},\\
f_{2}(y(t)) = & \beta y_{1} y_{2} -(d+r) y_{2},  
\quad t \in [0, T], \\
f_{3}(y(t)) = & r y_{2}- d y_{3} +u y_{1}.
\end{cases}
\]
The problem can be rewritten in a compact form as
\[
\begin{cases}
^{abc}_{0}D_{t}^{\alpha}y = A y + f(y(t)),\\
y(0)= y^{0},
\end{cases}
\]
where $^{abc}_{0}D_{t}^{\alpha}$ is the Atangana--Baleanu
fractional derivative of order $\alpha \in (0, 1)$ 
in the sense of Caputo with respect to time $t$. The symbol
$\triangle$ denotes the Laplacian
with respect to the spacial variables, defined on
$H^{2}(\Omega)\bigcap H_{0}^{1}(\Omega)$.


\section{Existence of solution}
\label{sec4}

Existence of solution is proved in the weak sense.

\begin{definition}
We say that $y$ is a weak solution to \eqref{eq1} if
\begin{equation}
\label{equa1}
\int_{\Omega} \, (^{abc}_{0}D_{t}^{\alpha} y) v dx
+ \lambda \int_{\Omega} \nabla y \nabla v dx
= \int_{\Omega} f(y) v dx
\end{equation}
for all $v \in H^{1}(\Omega)$.
\end{definition}

Integrating by parts, involving the ABC fractional-time derivative
(see \cite{djida}), and using a straightforward calculation, one
obtains the following result.

\begin{proposition}
\label{prop4}
Let $y, v \in {C^{\infty}}(\overline{Q_{T}})$. Then,
\begin{multline}
\label{eq8}
\int_{0}^{T} \, ^{abc}_{0}D_{t}^{\alpha}y \cdot v dt
=- \int_{0}^{T} \, ^{abc}_{T}D_{t}^{\alpha}v \cdot y dt
+ \frac{B(\alpha)}{1-\alpha} v(x, T)
\int_{0}^{T} y E_{\alpha, \alpha}[-\gamma (T-t)^{\alpha}] dt\\
- \frac{B(\alpha)}{1-\alpha} y(x, 0)
\int_{0}^{T} E_{\alpha, \alpha}[-\gamma t^{\alpha}] v dt.
\end{multline}
\end{proposition}

Using the boundary conditions \eqref{eq2},
we immediately get the following Corollary.

\begin{corollary}
\label{cor5}
Let $y, v \in {C^{\infty}}(\overline{Q_{T}})$. Then,
\begin{equation*}
\begin{split}
\int_{\Omega}&\int_{0}^{T}
\left (  ^{abc}_{0}D_{t}^{\alpha}y  - \lambda \triangle y \right ) v dx dt\\
& =  - \frac{B(\alpha)}{1-\alpha} \int_{\Omega} \int_{0}^{T} y(x, 0)
E_{\alpha, \alpha}[-\gamma t^{\alpha}] + \int_{0}^{T}\int_{\Omega}
y \left ( - ^{abc}_{T}D_{t}^{\alpha} v - \lambda \triangle v   \right ) dx dt\\
&\quad +\frac{B(\alpha)}{1-\alpha} \int_{\Omega} v(x, T)
\int_{0}^{T}  y E_{\alpha, \alpha}[-\gamma (T-t)^{\alpha}] dt dx.
\end{split}
\end{equation*}
\end{corollary}
We proceed similarly as in \cite{djida}. Let $V_{m}$ define a subspace 
of $H^{1}(\Omega)$ generated by the $w_{1}$, $w_{2}$, $\ldots$, $w_{m}$ 
space vectors of orthogonal eigenfunctions of the operator $\Delta$.
We seek $u_{m}: t \in (0,  T] \rightarrow u_{m}(t) \in V_{m}$,
solution of the fractional differential equation
\begin{equation*}
\begin{cases}
\displaystyle \int_{\Omega} \, ^{abc}_{0}D_{t}^{\alpha} y_{m} v dx
+ \int_{\Omega} \nabla y_{m} \nabla v dx
= (f(y_{m}), v) & \mbox{ for all } v \in V_{m}, \\
y_{m}(x, 0)=y_{0m} & \mbox{ for } x \in \Omega.
\end{cases}
\end{equation*}

To continue the proof of existence, we recall the following auxiliary result.

\begin{theorem}[See \cite{djida}]
\label{thm6}
Let $\alpha \in (0, 1)$. Assume that $f \in L^{2}(Q_{T})$,
$y_{0} \in L^{2}(\Omega)$. Let $(\cdot, \cdot)$ be the
scalar product in $L^{2}(\Omega)$ and $a(\cdot, \cdot)$
be the bilinear form in $H^{1}_{0}(\Omega)$ defined by
$$
a(\phi, \psi)= \int_{\Omega} \nabla \phi(x) \nabla \psi(x) dx
\quad \forall \phi, \psi \in H^{1}(\Omega).
$$
Then the problem
\begin{equation*}
\begin{cases}
\left(^{abc}_{0}D_{t}^{\alpha}y, v\right) + a(y(t), v)
= (f(t), v), & \mbox{ for all } t \in (0, T), \\
y(x, 0)=y_{0}, & \mbox{ for } x \in \Omega,
\end{cases}
\end{equation*}
has a unique solution $y \in L^{2}(0, T, H_{0}^{1}(\Omega))
\bigcap {C}(0, T, L^{2}(\Omega))$ given by
\begin{multline}
\label{eq13}
y(x,t)=
\sum_{j=1}^{+\infty}  \bigg( \zeta_{j} E_{\alpha}[-\gamma_{j}
t^{\alpha}]y^{0}_{j}+\frac{(1-\alpha)\zeta_{j}}{B(\alpha)} f_{j}(t)\\
+  K_{j}\int_{0}^{t}(t-s)^{\alpha-1} E_{\alpha, \alpha}\left[
-\gamma_{j}(t-s)^{\alpha}\right] f_{j}(s) ds  \bigg) w_{j},
\end{multline}
where $\gamma_{j}$ and $\zeta_{j}$ are constants.
Moreover, provided $y_{0} \in L^{2}(\Omega)$,
$y$ satisfies the inequalities
\begin{equation}
\label{eq14}
\|y\|_{ L^{2}(0, T, H_{0}^{1}(\Omega))}
\leq \mu_{1} ( \|y_{0}\|_{H_{0}^{1}(\Omega)} +\|f\|_{L^{2}(Q_{T})})
\end{equation}
and
\begin{equation}
\label{eq15}
\|y\|_{ L^{2}(\Omega))} \leq \mu_{2} ( \|y_{0}\|_{L^{2}(\Omega)} +\|f\|_{L^{2}(Q_{T})}),
\end{equation}
where $\mu_{1}$ and $\mu_{2}$ are positive constants.
\end{theorem}

Since $f(y_{m}) \in L^{2}(Q_{T})$, Theorem~\ref{thm6} implies that $y_{m}$ 
is given in a explicit form. The existence of a solution is obtained by using the
a priori estimate of Theorem~\ref{thm6} and the same arguments used to pass 
to the limit as those used by us below in the proof of Theorem~\ref{thm:3.1}.


\section{Existence of an optimal control}
\label{sec5}

We prove existence of an optimal control by using
minimizing sequences.

\begin{theorem}
\label{thm:3.1}
There exists at least an optimal solution
$y^*(u^*) \in L^{\infty}(Q_{T})$
satisfying \eqref{eq1}--\eqref{eq3} and minimizing \eqref{eq4}.
\end{theorem}

\begin{proof}
Let $(y^{n}, u^{n})$ be a minimizing sequence of
$J(y, u)$ such that
$$
\lim_{n \rightarrow +\infty} J(y^{n}, u^{n})=J(y^*, u^*)
= \inf J(y, u)
$$
with $u^{n}, u \in U_{ad}$ and $y^{n}= (y_{1}^{n}, y_{2}^{n}, y_{3}^{n})$ 
satisfying the corresponding system to \eqref{eq1}
\begin{equation*}
\begin{split}
^{abc}_{0}D_{t}^{\alpha}y^{n}- \lambda \triangle y^{n}
&=  f(y^{n})\, ,
\mbox{ in } Q_{T}= \Omega \times (0, T), \\
\frac{\partial y^{n}}{\partial\nu}
&= 0 \, ,  \mbox{ on } \Sigma_{T}
=\partial \Omega \times (0, T), \\
y^{n}(0, x)&= y^{0} \, ,  \mbox{ in } \Omega.
\end{split}
\end{equation*}
By Theorem~\ref{thm6}, we know that $(y_{n})$ is bounded,
independently of $n$, in $L^{2}(0, T, H^{1}(\Omega))$ 
and satisfying the inequalities
\begin{equation}
\label{eq17}
\|y^{n}\|_{ L^{2}(0, T, H_{0}^{1}(\Omega))}
\leq \mu_{1} ( \|y^{0}\|_{H_{0}^{1}(\Omega)} +\|f\|_{L^{2}(Q_{T})})
\end{equation}
and
\begin{equation}
\label{eq18}
\|y^{n}\|_{ L^{2}(\Omega))} 
\leq \mu_{2} ( \|y^{0}\|_{L^{2}(\Omega)} +\|f\|_{L^{2}(Q_{T})}),
\end{equation}
where $\mu_{1}$ and $\mu_{2}$ are positive constants. Then $(y^{n})$ 
is bounded in $L^{\infty}(0, T, L^{2}(\Omega))$ and $ L^{2}(0, T, H_{0}^{1}(\Omega))$.
By using the boundedness of $y_{i}^{n}$ ($|y_{i}|\leq N$, for $i=1,2,3$), 
the second member $f$ is in $L^{\infty}(Q_{T})$. Then, we have,
for a positive constant independent of $n$, that
\begin{equation*}
\| ^{abc}_{0}D_{t}^{\alpha}y^{n}
- \lambda \triangle y^{n} \|_{L^{2}(Q_{T})} \leq c.
\end{equation*}
Therefore, there exists  a subsequence of $y^{n}$,
still denoted by $(y^{n})$, and
$u^{n} \in U_{ad}$ such that
\begin{equation*}
\begin{gathered}
^{abc}_{0}D_{t}^{\alpha}y^{n}- \lambda \triangle y^{n}
\rightharpoonup \delta \mbox{ weakly in } L^{2}(Q_{T}),\\
y^{n} \rightharpoonup y^{*} \mbox{ weakly in } L^{2}(0, T,
H_{0}^{1}(\Omega)).
\end{gathered}
\end{equation*}
We now show that $\frac{\partial y^{n}}{\partial t}$ is bounded 
in $L^{1}(0, T, H^{-1}(\Omega))$.  We shall use the following lemma.

\begin{lemma}
\label{lm52}
If $u \in L^{\infty}(0, T, L^{2}(\Omega)) \cap H^{1}(0, T, L^{1}(\Omega))$, 
then there exists a positive constant $c$ such that
$$
\| \partial_{t}u \|_{L^{1}(0, T, L^{1}(\Omega))}
\leq \frac{c}{E_{\alpha}(-\gamma T^{\alpha})} 
\|u \|_{L^{\infty}(0, T, L^{2}(\Omega))}.
$$
\end{lemma}

\begin{proof}
Since for $0 \leq s \leq t \leq T$, $t \rightarrow E_{\alpha}(-t) $ 
is completely monotonic, we have
$$
E_{\alpha}(-\gamma T^{\alpha}) \leq E_{\alpha}(-\gamma (t-s)^{\alpha}).
$$
It yields that
$$
E_{\alpha}(-\gamma T^{\alpha}) \int_{0}^{t} |\partial_{s}u| ds 
\leq \int_{0}^{t} |\partial_{s}u|  E_{\alpha}(-\gamma (t-s)^{\alpha}) ds.
$$
Using the well-known inequality  $\| ^{abc}_{0}D{_t^{\alpha}} 
u \|_{L^{\infty}(0, T)} \leq \frac{B(\alpha)}{1-\alpha} 
\|u\|_{L^{\infty}(0, T)}$, we get
\begin{equation}
\begin{aligned}
E_{\alpha}(-\gamma T^{\alpha})\frac{B(\alpha)}{1-\alpha} 
\int_{0}^{t} |\partial_{s}u| ds 
& \leq \frac{B(\alpha)}{1-\alpha} \int_{0}^{t} |\partial_{s}u|  
E_{\alpha}(-\gamma (t-s)^{\alpha}) ds \\
& \leq \| ^{abc}_{0}D{_t^{\alpha}} u \|_{L^{\infty}(0, T)} 
\leq \frac{B(\alpha)}{1-\alpha} \|u\|_{L^{\infty}(0, T)}.
\end{aligned}
\end{equation}
It follows that
$$
\int_{0}^{T} |\partial_{s}u| ds 
\leq \frac{1}{E_{\alpha}(-\gamma T^{\alpha})} \|u\|_{L^{\infty}(0, T)}.
$$
Integrating over $\Omega$, we have
$$
\int_{0}^{T}\int_{\Omega} |\partial_{s}u| ds dx  
\leq \frac{1}{E_{\alpha}(-\gamma T^{\alpha})} 
\int_{\Omega}\|u\|_{L^{\infty}(0, T)}dx.
$$
Then, for a positive constant $c$, one has
\begin{equation}
\begin{aligned}
\| \partial_{t}u \|_{L^{1}(0, T, L^{1}(\Omega))} 
& \leq \frac{1}{E_{\alpha}(-\gamma T^{\alpha})} 
\|u \|_{L^{\infty}(0, T, L^{1}(\Omega))} \\
& \leq \frac{c}{E_{\alpha}(-\gamma T^{\alpha})} 
\|u \|_{L^{\infty}(0, T, L^{2}(\Omega))}.
\end{aligned}
\end{equation}
The proof is complete.
\end{proof}

By the estimate \eqref{eq18} of $y^{n}$ and Lemma~\ref{lm52}, 
we have that $\partial_{t} y^{n}$ is bounded in $L^{1}(0, T, L^{1}(\Omega))$.
Due to \eqref{eq17}, we have that $y^{n}$ is bounded 
in $L^{2}(0, T, H_{0}^{1}(\Omega))$. Set
$$
W = \{ v \in L^{2}(0, T, H_{0}^{1}(\Omega)), 
\partial_{t} v \in  L^{1}(0, T, L^{1}(\Omega)) \}.
$$
Using the classical argument of Aubin, the space $W$ is compactly 
embedded in $L^{2}(0, T, L^{2}(\Omega))= L^{2}(Q_{T})$. We can 
then extract a subsequence from $y^{n}$, not relabeled, such that
\begin{equation*}
\begin{gathered}
y^{n} \rightharpoonup y^{*} \mbox{ weakly in } L^{\infty}(0, T,
L^{2}(\Omega)) \mbox{ and in }  L^{2}(Q_{T}),\\
y^{n} \rightarrow y^{*} \mbox{ strongly in } L^{2}(Q_{T}),\\
y^{n} \rightarrow y^{*} \mbox{ a. e.  in }L^{2}(Q_{T}),\\
y^{n}(T) \rightarrow y^{*}(T) \mbox{ in } L^{2}(\Omega).
\end{gathered}
\end{equation*}
Denote $\mathbb{D}'(Q_{T})$ the dual of $\mathbb{D}(Q_{T})$,
the set of ${C}^{\infty}$ functions on $Q_{T}$ with compact support.
We claim that
$$
^{abc}_{0}D_{t}^{\alpha}y^{n}- \lambda \triangle y^{n}
\rightharpoonup  {^{abc}_{0}D_{t}^{\alpha}y^{*}}
- \lambda \triangle y^{*} \mbox{ weakly in } \mathbb{D}'(Q_{T}).
$$
Indeed, we have
$$
\int_{0}^{T} \int_{\Omega} y^{n}( ^{abc}_{0}D_{t}^{\alpha}v
- \lambda \triangle v ) dx dt
\rightarrow
\int_{0}^{T} \int_{\Omega} y^{*}( -^{abc}_{T}D_{t}^{\alpha}v -
\lambda  \triangle v ) dx dt, \forall v \in \mathbb{D}(Q_{T})
$$
and
$$
\int_{\Omega} v(x, T) \int_{0}^{T} y^{n}
E_{\alpha, \alpha}[-\gamma (T-t)^{\alpha}] dt dx
\rightarrow
\int_{\Omega} v(x, T) \int_{0}^{T} y^{*} E_{\alpha, \alpha}[
-\gamma (T-t)^{\alpha}] dt dx.
$$
On the other hand, the convergence 
$y_{i}^{n} \rightarrow y_{i}^{*}$ in $L^{2}(Q_{T})$
and the essential boundedness of $y_{1}^{n}$ and $y_{2}^{n}$ 
imply $y_{1}^{n} y_{2}^{n} \rightarrow y_{1}^{*}y_{2}^{*}$ 
in $L^{2}(Q_{T})$. Modulo a subsequence denoted $u^{n}$, we have
$$
u^{n} \rightharpoonup u^{*} \text{ weakly  in } L^{2}(Q_{T}).
$$
We deduce that $u^{*} \in U_{ad}$ as a consequence of the closure 
and the boundedness of this set in $L^{2}(Q_{T})$ 
and thus it is weakly closed. Similarly, we can prove that
$$
u^{n}y_{1}^{n}  \rightarrow u^{*}y_{1}^{*} \text{ in } L^{2}(Q_{T}).
$$
Therefore,
$$
^{abc}_{0}D_{t}^{\alpha}y^{n}- \lambda \triangle y^{n}
\rightarrow ^{abc}_{0}D_{t}^{\alpha}y^{*}
- \triangle y^{*} \mbox{ weakly in } \mathbb{D}'(Q_{T}).
$$
From the uniqueness of the limit, we have
$$
^{abc}_{0}D_{t}^{\alpha}y^{*} - \triangle y^{*} = \delta.
$$
By passing to the limit as $n \rightarrow \infty$ 
in the equation satisfied by $y^{n}$,
we deduce that $y^{*}$ is a solution
of \eqref{eq1}. Finally,
the lower semi-continuity of $J$ leads to
$J(y^{*}, u^{*})  = \inf J(y, u)$.
Therefore, $y^{*}(u^{*})$ is an optimal solution.
\end{proof}


\section{Necessary optimality conditions}
\label{sec6}

In this section, our aim is to obtain optimality conditions.
As we shall see, our necessary optimality conditions involve
an adjoint system defined by means of the backward Atangana--Baleanu
fractional-time derivative.

Let $y^{*}(u^{*})$ be an optimal solution and $u^{\varepsilon}
= u^{*} + \varepsilon u \in U_{ad}$ be a control function such 
that $u \in U_{ad}$ and $\varepsilon >0$.
Denote $y^{\varepsilon}= (y_{1}^{\varepsilon}, 
y_{2}^{\varepsilon}, y_{3}^{\varepsilon})= (y_{1}, y_{2}, y_{3}) (u^{\varepsilon})$ 
and $y^{*}= (y_{1}^{*}, y_{2}^{*}, y_{3}^{*})= (y_{1}, y_{2}, y_{3}) (u^{*})$ 
the solutions of \eqref{eq1}--\eqref{eq3} corresponding to $u^{\varepsilon}$ 
and $u^{*}$, respectively. Setting $y^{\varepsilon}= y^{*} 
+ \varepsilon z^{\varepsilon}$ and subtracting the system corresponding to $y^{*}$ 
from the one corresponding to $y^{\varepsilon}$, we have
\begin{equation}
\label{eq16}
^{abc}_{0}D_{t}^{\alpha}\left(\frac{y^{\varepsilon}-y^{*}}{\varepsilon}\right)
- \lambda \triangle\left(\frac{y^{\varepsilon}-y^{*}}{\varepsilon}\right)
= \dfrac{f(y^{\varepsilon})-f(y^{*})}{\varepsilon}.
\end{equation}
System \eqref{eq16} can be rewritten as
\begin{equation*}
^{abc}_{0}D_{t}^{\alpha}z^{\varepsilon}
- \lambda \triangle  z^{\varepsilon}
= \dfrac{f(y^{\varepsilon})-f(y^{*})}{\varepsilon},
\end{equation*}
associated to Neumann boundary conditions
\begin{equation*}
\frac{\partial z_{1}^{\varepsilon}}{\partial \nu}
= \frac{\partial z_{2}^{\varepsilon}}{\partial \nu}
= \frac{\partial z_{3}^{\varepsilon}}{\partial \nu}=0 
\text{ on } \Sigma_{T}
\end{equation*}
and initial condition
$$
z^{\varepsilon}=0 \text{ in } \Omega \, ,
$$
where $z^{\varepsilon}= ( z_{1}^{\varepsilon}, 
z_{2}^{\varepsilon}, z_{3}^{\varepsilon})$ and
\[
\dfrac{f(y^{\varepsilon})-f(y^{*})}{\varepsilon}
=
\begin{cases}
\dfrac{f_{1}(y^{\varepsilon})- f_{1}(y^{*})}{{\varepsilon}}
= & (\mu- \beta y_{2}^{\varepsilon}-d - u^{\varepsilon})
z_{1}^{\varepsilon}-(\beta y_{1}^{*} + \mu) z_{2}^{\varepsilon}\\ 
&\quad + \mu z_{3}^{\varepsilon} - u y_{1}^{*},\\
\dfrac{f_{2}(y^{\varepsilon})- f_{2}(y^{*})}{{\varepsilon}}
=& \beta y_{2}^{\varepsilon} z_{1}^{\varepsilon} 
+ ( \beta y_{1}^{*} -d-r) z_{2}^{\varepsilon}, \\
\dfrac{f_{3}(y^{\varepsilon})- f_{3}(y^{*})}{{\varepsilon}} 
=& u^{\varepsilon} z_{1}^{\varepsilon}+ rz_{2}^{\varepsilon}
- dz_{3}^{\varepsilon}+u y_{1}^{*}.
\end{cases}
\]
Set
\[
F^{\varepsilon}=
\left(
\begin{array}{ccc}
\mu- \beta y_{2}^{\varepsilon}-d - u^{\varepsilon}
& -\beta y_{1}^{*}+ \mu  & \mu \\
\beta y_{2}^{\varepsilon} 
&  \beta y_{1}^{*} -d-r & 0 \\
u^{\varepsilon} & r & -d \\
\end{array}
\right)
\]
and
\[
G=
\left(
\begin{array}{c}
-y_{1}^{*} \\
0 \\
y_{1}^{*} \\
\end{array}
\right).
\]
Then, \eqref{eq16} can be reformulated in the following form:
\[
\begin{cases}
^{abc}_{0}D_{t}^{\alpha} z^{\varepsilon}
- \lambda \triangle  z^{\varepsilon}
&= F^{\varepsilon} z^{\varepsilon}+ Gu \text{ for } t\in[0, T],\\
z^{\varepsilon}(0)&= 0.
\end{cases}
\]
Since the elements of the matrix $F^{\varepsilon}$ are uniformly bounded 
with respect to $\varepsilon$ and $(-u y_{1}^{*}, 0, u y_{1}^{*})$ 
is bounded in $L^{2}(Q_{T})$, it follows by Theorem~\ref{thm6} 
that $z^{\varepsilon}= \frac{y^{\varepsilon}-y^{*}}{\varepsilon}$ 
is bounded in $L^{\infty}(0, T, L^{2}(\Omega)) 
\bigcap L^{2}(0, T, H^{1}(\Omega))$. Therefore, 
up to a subsequence of $z^{\varepsilon}$,
there exists $z$ such that as $\varepsilon$ 
tends to zero we have
\begin{equation}
\label{eq23}
\begin{gathered}
z^{\varepsilon} \rightarrow z \mbox{ weakly in } 
L^{\infty}(0, T, L^{2}(\Omega)) \mbox{ and in } L^{2}(0, T, H^{1}(\Omega)),\\
\frac{\partial z^{\varepsilon}}{\partial t}\rightarrow 
\frac{\partial z}{\partial t} \mbox{ weakly in } L^{2}(0, T, H^{-1}\Omega)).
\end{gathered}
\end{equation}
Put
\[
F=
\left(
\begin{array}{ccc}
\mu- \beta y_{2}^{*}-d - u^{*}& -\beta y_{1}^{*}+ \mu  & \mu \\
\beta y_{2}^{*} &  \beta y_{1}^{*} -d-r & 0 \\
u^{*} & r & -d \\
\end{array}
\right).
\]
Note that all the components of the matrix $F^{\varepsilon}$ 
tend to the corresponding ones of the matrix $F$ in $L^{2}(Q_{T})$ 
as $\varepsilon \rightarrow 0$. From equations satisfied by 
$y_{\varepsilon}$ and $y$, we have that
\begin{equation*}
\int_{\Omega}\int_{0}^{T} \, ^{abc}_{0}D_{t}^{\alpha}
z^{\varepsilon} v dx dt
+ \int_{\Omega}\int_{0}^{T} \lambda \nabla z^{\varepsilon}
\nabla v dx dt
= \int_{\Omega}\int_{0}^{T} \, 
\dfrac{(f(y^{\varepsilon})-f(y))}{\varepsilon} v dx dt.
\end{equation*}
Letting $\varepsilon \rightarrow 0$, we get
\begin{equation*}
\int_{\Omega}\int_{0}^{T} \, ^{abc}_{0}D_{t}^{\alpha}
z v dx dt
+ \int_{\Omega}\int_{0}^{T} \lambda \nabla z
\nabla v dx dt
= \int_{\Omega}\int_{0}^{T} \, (F z+ Gu) v dx dt
\end{equation*}
with $z(0)=0$.  By Green's formula, it follows that
\begin{multline*}
\int_{\Omega}\int_{0}^{T} \, ^{abc}_{0}D_{t}^{\alpha} z \cdot v dx dt
- \int_{\Omega}\int_{0}^{T} \lambda \Delta z \cdot v  dx dt
+ \int_{0}^{T}\int_{\partial \Omega} \frac{\partial z}{\partial \eta} v ds dt\\
=  \int_{\Omega}\int_{0}^{T} \, (F z+ Gu) v dx dt.
\end{multline*}
Then $z$ verifies
\begin{equation*}
\begin{gathered}
^{abc}_{0}D_{t}^{\alpha} z - \lambda \triangle z
= F z+ Gu, \mbox{ in } \Omega,  \\
\frac{\partial z}{\partial \nu}  = 0
\quad \mbox{ on } \partial \Omega,\\
z(0)=0.
\end{gathered}
\end{equation*}
To derive the adjoint operator associated with $z$, 
we need to introduce an enough smooth adjoint 
variable $p$ defined in $Q_{T}$. We have
\begin{equation*}
\int_{\Omega}\int_{0}^{T} \left(  ^{abc}_{0}D_{t}^{\alpha}z
- \lambda \triangle z \right) p dx dt
=\int_{Q_{T}} (F z+ Gu) p dx dt.
\end{equation*}
Integrating by parts, one has
\begin{equation*}
\begin{split}
\int_{\Omega}\int_{0}^{T} \, ^{abc}_{0}D_{t}^{\alpha}z. p dt dx
&= -  \int_{\Omega} \int_{0}^{T} \, ^{abc}_{T}D_{t}^{\alpha}p. z dt dx\\
&\quad + \frac{B(\alpha)}{1-\alpha} \int_{\Omega} p(x, T)
\int_{0}^{T} z E_{\alpha, \alpha}[-\gamma (T-t)^{\alpha}] dt.
\end{split}
\end{equation*}
We conclude that the adjoint function $p$ 
satisfies the adjoint system given by
\begin{equation}
\label{eq10}
\begin{gathered}
- ^{abc}_{T}D_{t}^{\alpha}p - \lambda \triangle p -F p
= D^{*}D y^{*}, \, t \in [0, T], \\
\frac{\partial p}{\partial \nu} = 0
\mbox{ on } \partial \Omega \times (0, T),\\
p(T, x) = D^{*}D y^{*}(T, x),
\end{gathered}
\end{equation}
where $D$ is the matrix defined by
$$
\left(
\begin{array}{ccc}
0 & 0 & 0 \\
0 & 1 & 0 \\
0 & 0 & 0 \\
\end{array}
\right).
$$

Similarly to the existence result of Theorem~\ref{thm6}, 
one can show that the solution of the adjoint 
system \eqref{eq10} exists.

\begin{theorem}
\label{thm}
Given an optimal control $u^{*}$ and corresponding state $y^{*}$, 
there exists a solution $p$ to the adjoint system. Furthermore, $u^{*}$
can be characterized, in explicit form, as
\begin{equation}
\label{eq11}
\begin{split}
u^{*}= & \min \left( 1, \max \left(0, -\frac{1}{\theta}G^{*}p\right) \right) \\
= &  \min \left( 1, \max \left(0, -\frac{y_{1}^{*}}{\theta}(p_{1}-p_{3})\right) \right).
\end{split}
\end{equation}
\end{theorem}

The proof of Theorem~\ref{thm} is classical and  follows exactly 
the same arguments as in \cite{ref27}, using the fact that the minimum 
of the objective function $J$ is achieved at $u^{*}$. For small $\varepsilon$ 
such that $u^{\varepsilon}= u^{*} +\varepsilon h \in U_{ad}$, one can prove that
\begin{equation*}
\begin{split}
\frac{J(u^{*} +\varepsilon h)-J(u)}{\varepsilon}& \geq 0,
\end{split}
\end{equation*}
equivalent to
\begin{equation*}
\begin{split}
\int_{0}^{T} (G^{*}p+ \theta u^{*}, h ) dt &\geq 0 
\quad \forall h \in U_{ad}.
\end{split}
\end{equation*}
The characterization result is obtained 
by standard arguments of variations of $h$.


\section{Numerical results}
\label{sec7}

In this section, we study the effect of the order of differentiation $\alpha$ 
to the dynamic of infection in space during a given time interval. 
We can mention two cases: absence and presence of vaccination.  
In the following, we consider a domain of $10km^2$ square grid, 
which represents a city for the population under consideration. 
We assume that the infection originates in the subdomain 
$\Omega_1=cell(1,1)$ when the disease starts at the lower left corner 
of $\Omega$. At $t =0$, we assume that the susceptible people are 
homogeneously distributed with $50$ in each $1km^2$ cell 
except at the subdomain $\Omega_1$ of $\Omega$, where we introduce $7$ 
infected individuals and keep $43$ susceptible there.  
The parameters and initial values are given in Table~\ref{tab1}.

\begin{table}	
\caption{Values of initial conditions and parameters.}
\label{tab1}
\centering
\begin{tabular}{C{2.5cm} C{4cm} C{4cm}} \hline 
Symbol & Description (Unit) &  Value  \\ \hline \hline  
$S_0(x,y)$ & Initial susceptible population $(people/km^2)$ 
& $43$ for $(x,y)\in\Omega_1$ \hspace{1.5cm}  $50$ for $(x,y)\notin\Omega_1$ \\ \hline  
$I_0(x,y)$ & Initial infected population $(people/km^2)$ & $7$ for $(x,y)\in\Omega_1$ 
\hspace{1.5cm} $0$ for $(x,y)\notin\Omega_1$  \\ \hline  
$R_0(x,y)$ & Initial recovered population $(people/km^2)$ 
& $0$ for $(x,y)\in\Omega_1$ \hspace{1.5cm} $0$ for $(x,y)\notin\Omega_1$\\ \hline 
$\lambda_1=\lambda_2=\lambda_3$ & Diffusion coefficient ($km^2/day$) & 0.6 \\
\hline $\mu$ & Birth rate $(day^{-1})$ & 0.02 \\
\hline $d$ & Natural death rate $(day^{-1})$ & 0.03 \\
\hline $\beta$ & Transmission rate $((people/km^2)^{-1}.day^{-1})$ & 0.9 \\
\hline $r$ & Recovery rate $(day^{-1})$ & 0.04 \\
\hline $T$ & Final time $(day)$ & 20 \\	\hline
\end{tabular}
\end{table}

We have used MATLAB to implement the so-called forward-backward 
sweep method \cite{ref28} to solve our fractional optimal 
control problem \eqref{eq1}--\eqref{eq5}. The state system 
and the adjoint equations are numerically integrated using 
an approximation of the (left/right) ABC fractional derivative, 
based on a explicit finite difference method \cite{ref25,ref29}.  
The algorithm can be summarized as follows:
\begin{algorithm}
\caption{Forward-backward sweep method}
\begin{algorithmic}[1]
\STATE Set $n$ the number of subdivisions, $h$ the step size, 
$m$ the number of time steps, $\tau$ the step time, 
$\delta = 0.001$ the tolerance, and $test = -1$.
\STATE Initiate the control $u_{old}$, the state 
$(S_{old},I_{old},R_{old})$ and adjoint $((p_{old})_1,(p_{old})_2,(p_{old})_3)$.
\WHILE {$test < 0$}
\STATE Solve the state equation \eqref{eq1} for $(S,I,R)$ with initial 
guess $(S_{0},I_{0},R_{0})$, using  an explicit finite 
difference method forward in time.
\STATE Solve the adjoint equation \eqref{eq10} for $(p_1,p_2,p_3)$ 
using the transversality conditions $p_i(T)$ and $(S,I,R)$ backward in time.
\STATE Update the control using the gradient equation \eqref{eq11} to reach $u$.
\STATE Compute the tolerance criteria $\psi_1=\delta\|S\|-\|S-S_{old}\|$, 
$\psi_2=\delta\|I\|-\|I-I_{old}\|$, 
$\psi_3=\delta\|R\|-\|R-R_{old}\|$, 
$\psi_4=\delta\|p_1\|-\|p_1-(p_{old})_1\|$, 
$\psi_5=\delta\|p_2\|-\|p_2-(p_{old})_2\|$, 
$\psi_6=\delta\|p_3\|-\|p_3-(p_{old})_3\|$, 
$\psi_7=\delta\|u\|-\|u-u_{old}\|$, and calculate 
$test=\min\{\psi_1,\psi_2,\psi_3,\psi_4,\psi_5,\psi_6,\psi_7\}$.
\ENDWHILE
\end{algorithmic}
\end{algorithm}	


\subsection{Fractional $\alpha$-dynamics without control}

Figures~\ref{fig01}, \ref{fig02} and \ref{fig03} present the numerical 
results with different values of $\alpha$ in the case of absence of control. 
We observe that the susceptible individuals are transferred to the infected class 
while the disease spreads from the lower left corner to the upper right corner. 
In Figure~\ref{fig01}, for $\alpha=1$, we can see that the epidemic takes 
$20$ days to cover the entire area (50 infected per cell in all $\Omega$), 
but in Figures~\ref{fig02} and \ref{fig03} this is not the case 
for $\alpha=0.95$ and $\alpha=0.9$. It is clear that the number of individuals 
infected is almost $44$ per cell in the upper right corner.
\begin{figure}[ht!]
\centering
\includegraphics[scale=0.35]{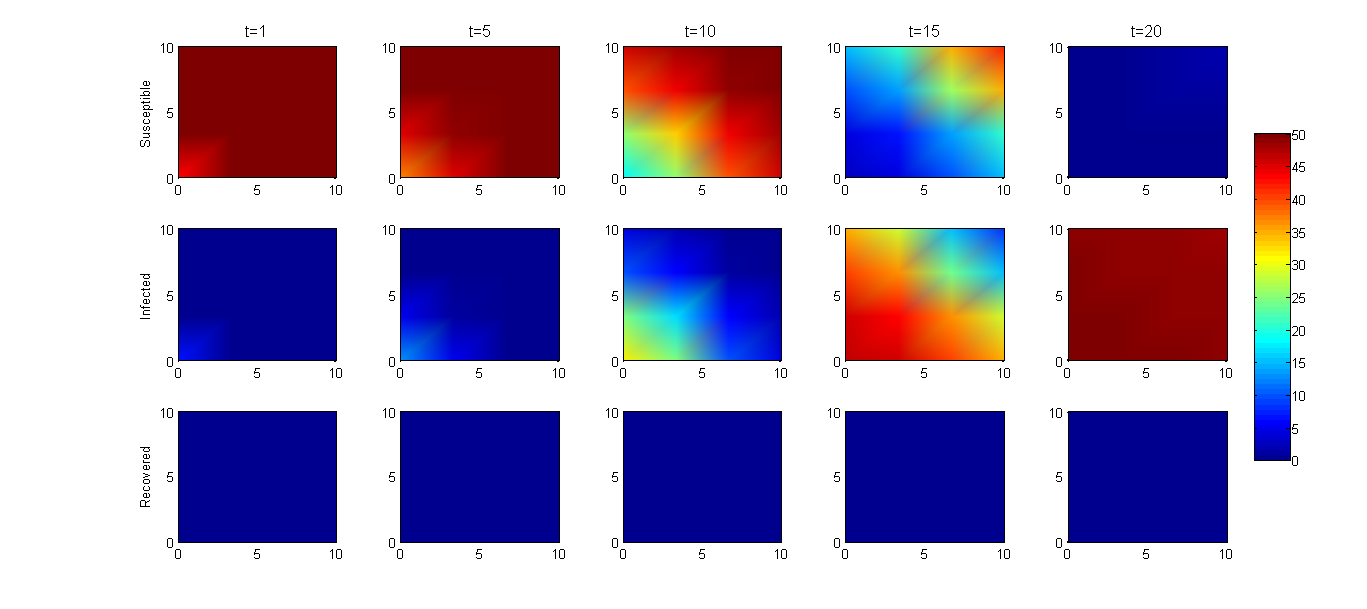}
\caption{Dynamic of the system without control for $\alpha=1$.}
\label{fig01}
\end{figure}
\begin{figure}[ht!]
\includegraphics[scale=0.35]{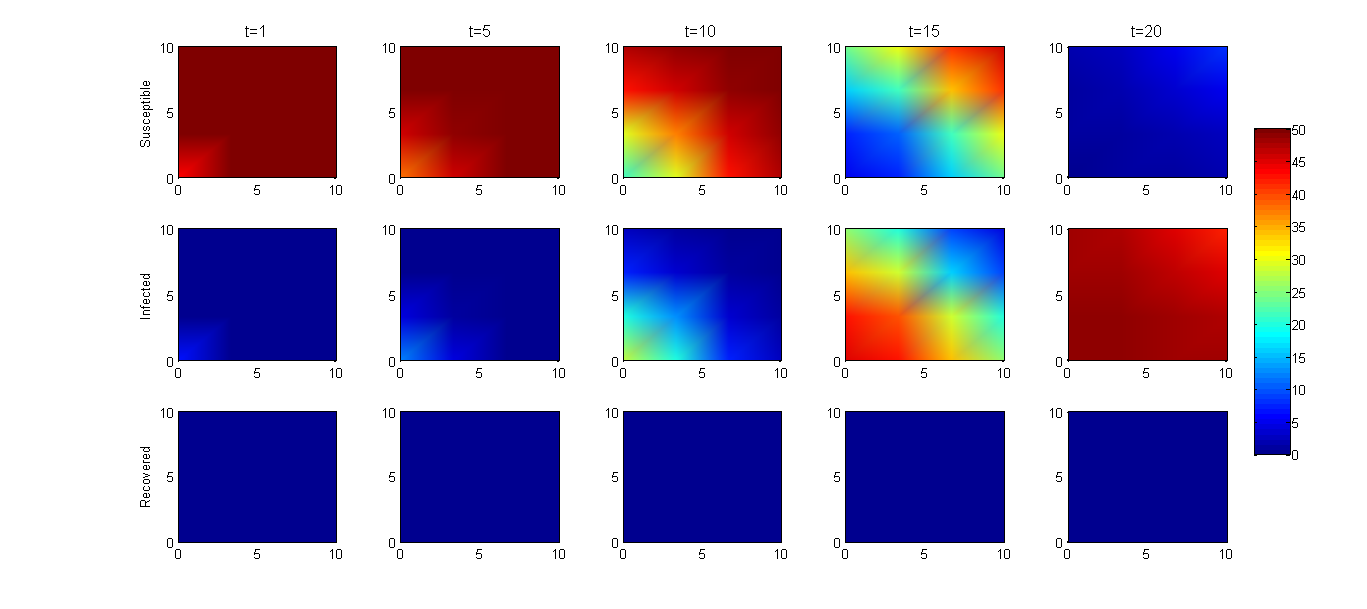}
\caption{Dynamic of the system without control for $\alpha=0.95$.}
\label{fig02}
\end{figure}
\begin{figure}[ht!]
\centering
\includegraphics[scale=0.35]{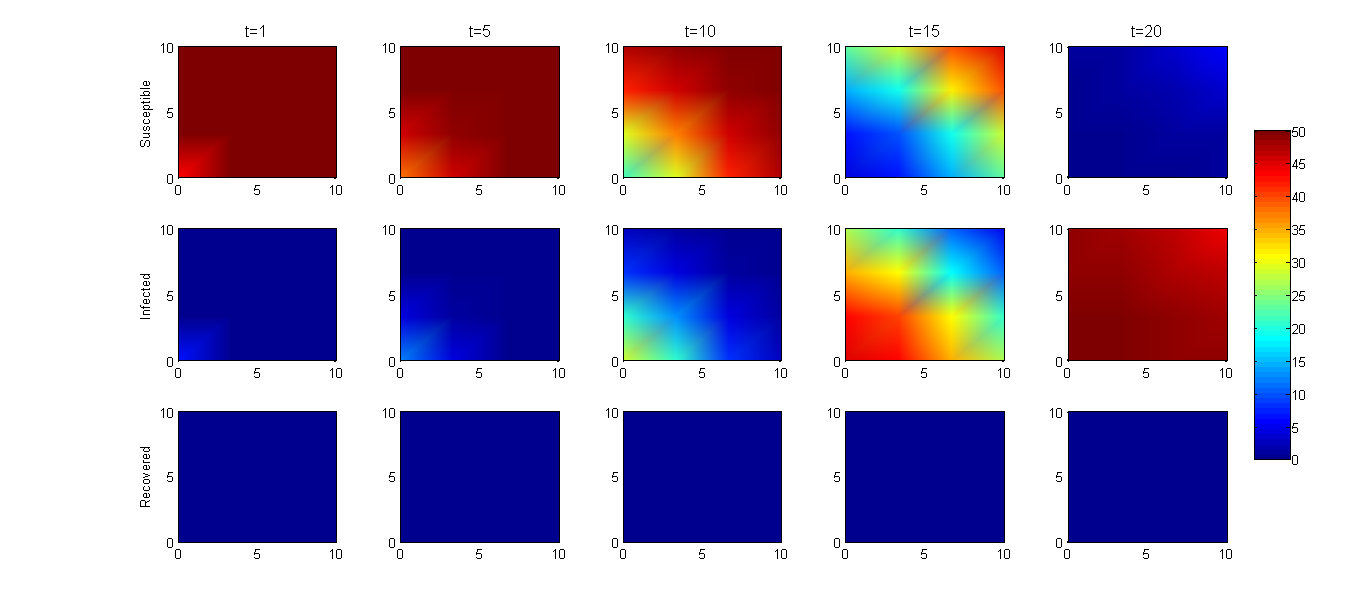}
\caption{Dynamic of the system without control for $\alpha=0.9$.}
\label{fig03}
\end{figure}


\subsection{Fractional $\alpha$-dynamics with optimal vaccination strategy}

We compare the infection prevalence over a period of $20$ days 
in the presence of the vaccination strategy. We note that the 
susceptible individuals are transferred to the recovered class 
(see Figures~\ref{fig04}, \ref{fig05} and \ref{fig06}). 
In Figure~\ref{fig04}, we see that the number of infected people 
is $40$ per cell and $10$ per cell for recovered individuals. 
In Figures~\ref{fig05} and \ref{fig06}, we have almost $5$ 
susceptible people per cell, $35$ infected people per cell, 
and $10$ recovered individuals per cell.
\begin{figure}[ht!]
\centering
\includegraphics[scale=0.35]{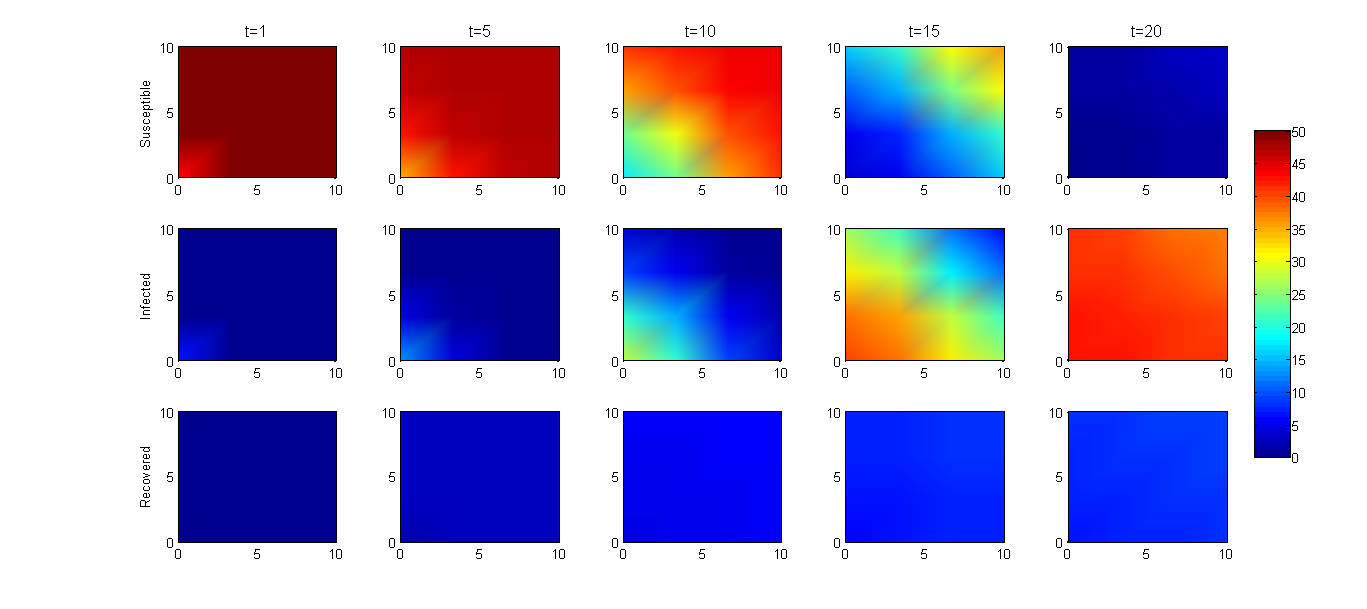}
\caption{Dynamic of the system with control for $\alpha=1$.}
\label{fig04}
\end{figure}
\begin{figure}[ht!]
\includegraphics[scale=0.35]{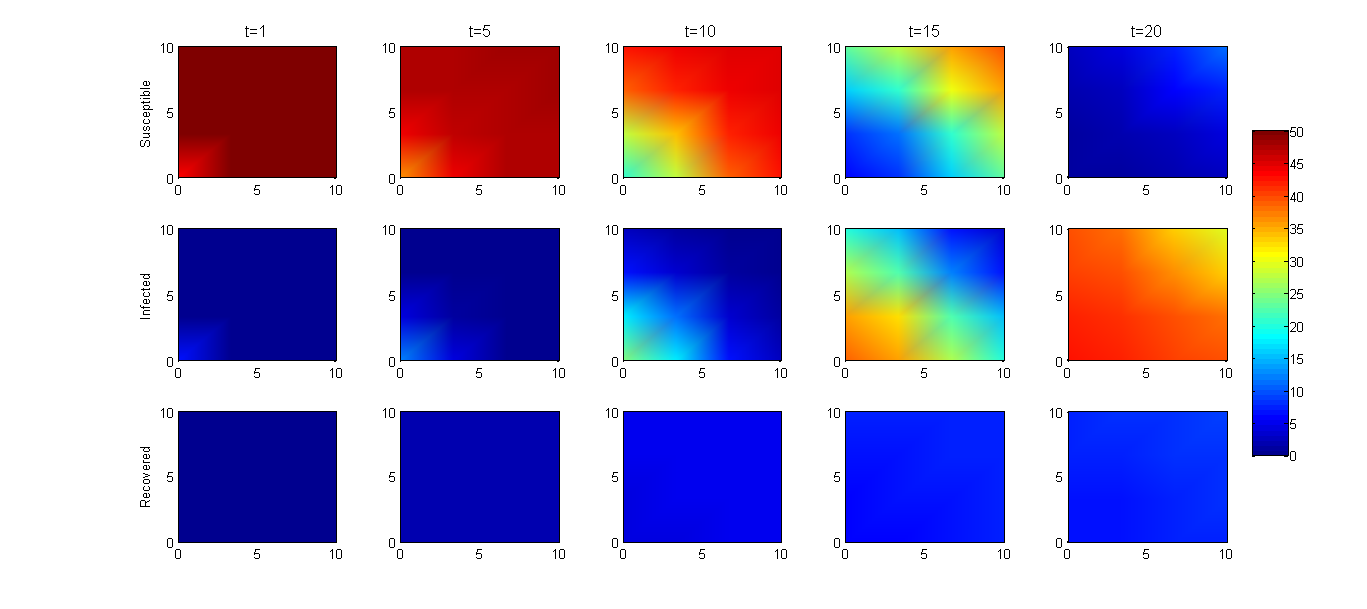}
\caption{Dynamic of the system with control for $\alpha=0.95$.}
\label{fig05}
\end{figure}
\begin{figure}[ht!]
\centering
\includegraphics[scale=0.35]{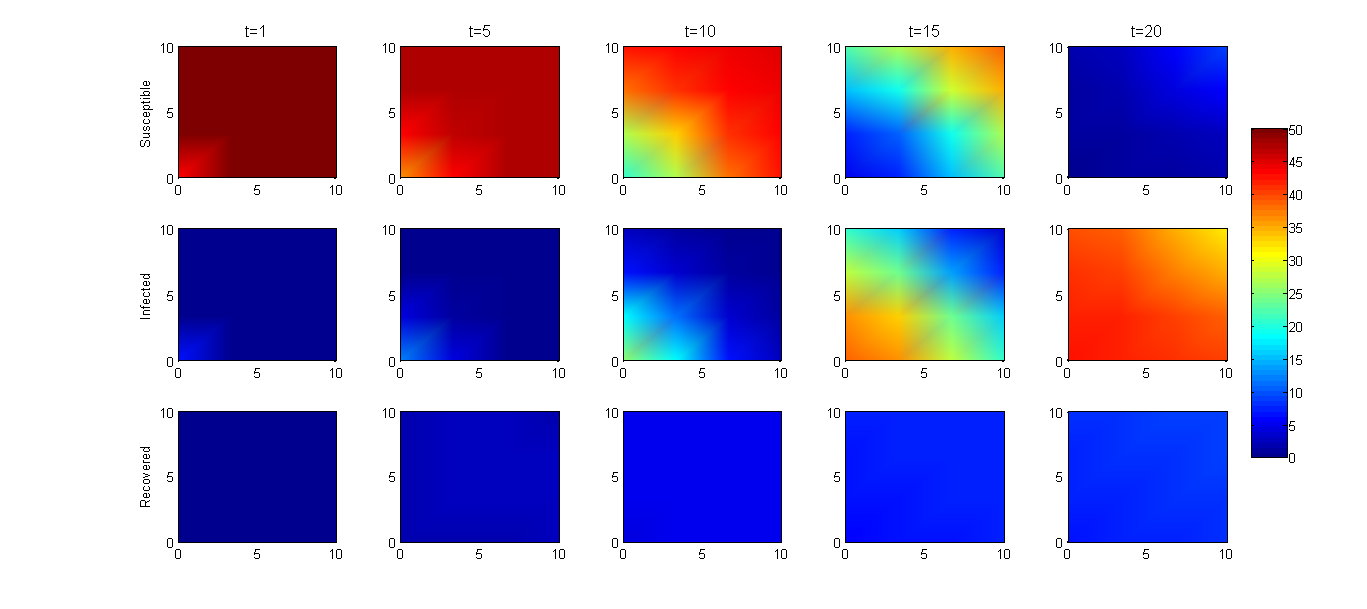}
\caption{Dynamic of the system with control for $\alpha=0.9$.}
\label{fig06}
\end{figure}
Next, we investigate the effect of the order $\alpha$ 
to the value of the cost functional $J$ in absence and presence 
of vaccination. Before that, we present the results 
in Table ~\ref{tab2} and Table ~\ref{tab3}, respectively.

\begin{table}	
\caption{Values of the cost functional $J$ 
without control for different $\alpha$.}
\label{tab2}
\centering
\begin{tabular}{C{1.5cm} C{3cm} C{3cm} C{3cm}}
\hline $\alpha$ & 0.9 &  0.95 & 1  \\
\hline J & $7.4350 e^{+04}$ & $7.1586 e^{+04}$ & $7.7019 e^{+04}$  \\
\hline
\end{tabular}
\end{table}

\begin{table}	
\caption{Values of the cost functional $J$ 
with control for different $\alpha$.}
\label{tab3}
\centering
\begin{tabular}{C{1.5cm} C{3cm} C{3cm} C{3cm}}
\hline $\alpha$ & 0.9 &  0.95 & 1  \\
\hline J & $4.9157 e^{+04}$ & $4.7489 e^{+04}$ & $5.2503 e^{+04}$  \\
\hline
\end{tabular}
\end{table}

We note that the functional $J$ decreases under the effect of vaccination 
for different values of $\alpha$, and the value of $J$ is optimal 
as $\alpha=0.95$. Furthermore, the results obtained in fractional order 
cases show that the spread of the disease takes more than 20 days 
to cover the entire space with the same cost of the vaccination 
strategy in the case of integer derivatives.


\section{Conclusion}
\label{sec8}

In this study, we investigated the optimal vaccination strategy 
for a fractional SIR model. Interactions between susceptible, 
infected, and recovered are modeled by a system of partial 
differential equations with Atangana--Baleanu--Caputo  
fractional time derivative. We proved existence of solutions 
to our fractional parabolic state system as well as the 
existence of an optimal control. For a given objective functional 
$J$, an optimal control is characterized in terms of the corresponding 
state and adjoint variables. In order to control the infection, 
we have compared the dynamics of our system with different 
values of $\alpha$. We noticed that the values of $J$ decreases 
under the effect of vaccination for different values of $\alpha$.
Moreover, with the presence of an optimal vaccination strategy, 
we found that the smallest value of the cost-functional $J$ 
is obtained when $\alpha=0.95$. Then, an analysis of the proposed 
fractional order strategy with a well chosen  fractional order $\alpha$ 
shows that it is more cost-effective than the classical strategy. 
Finally, the results obtained when $\alpha$ takes a fractional value 
show that the spread of the disease takes more than $20$ days 
to cover the entire space with the same cost of the vaccination 
strategy in the case of integer-order derivatives.


\section*{Acknowledgments}

Torres is supported by Funda\c{c}\~ao 
para a Ci\^encia e a Tecnologia (FCT),
within project UIDB/04106/2020 (CIDMA).


\renewcommand{\bibname}{References}


\medskip


\end{document}